\documentclass[11pt]{amsart}

\usepackage{amssymb}
\usepackage{amsmath}
\usepackage{amsthm}
\usepackage{enumerate}
\usepackage{graphicx}

\theoremstyle{plain}

\newtheorem{theorem}{Theorem}[section]

\newtheorem{lemma}[theorem]{Lemma}
\newtheorem{proposition}[theorem]{Proposition}
\newtheorem{corollary}[theorem]{Corollary}

\newtheorem{observation}[theorem]{Observation}
\newtheorem{fact}[theorem]{Fact}

\theoremstyle{definition}

\newtheorem{definition}[theorem]{Definition}
\newtheorem{remark}[theorem]{Remark}

\newtheorem*{problem*}{Problem}
\newtheorem{problem}[theorem]{Problem}

\newcommand{\R}{\mathbb{R}}
\newcommand{\N}{\mathbb{N}}

\newcommand{\D}{\mathrm{D}}

\newcommand{\inte}{\mathrm{int}}
\newcommand{\cconv}{\overline{\mathrm{conv}}\,}
\newcommand{\conv}{{\mathrm{conv}}\,}
\newcommand{\ext}{\mathrm{ext}}

\renewcommand{\epsilon}{\varepsilon}
\renewcommand{\phi}{\varphi}

\newcommand{\B}{\mathcal{B}}

\begin{document}

\title{\large Unit balls of Polyhedral Banach spaces with many extreme points}

\author{Carlo Alberto De Bernardi
}

\address{Dipartimento di Matematica per le Scienze economiche, finanziarie ed attuariali, Universit\`{a} Cattolica del Sacro Cuore, Via Necchi 9, 20123 Milano, Italy}

\email{carloalberto.debernardi@unicatt.it}
\email{carloalberto.debernardi@gmail.com}



 \thanks{}

\dedicatory{}

\begin{abstract} Let $E$ be a  $(\mathrm{IV})$-polyhedral Banach space. We show that, for each $\epsilon>0$, $E$ admits an $\epsilon$-equivalent  $\mathrm{(V)}$-polyhedral norm  such that the corresponding closed unit ball  is the closed
convex hull of its extreme points. 

In particular, we obtain that every separable isomorphically  polyhedral Banach space,
for each $\epsilon>0$,  admits an $\epsilon$-equivalent  $(\mathrm{V})$-polyhedral norm  such that the corresponding closed unit ball  is the closed
convex hull of its extreme points.
\end{abstract}

\maketitle


\markboth{	C.A.~De Bernardi
}{Unit balls of Polyhedral Banach spaces with many extreme points}

\section{Introduction}

Let $E$  be a real infinite-dimensional Banach space with topological dual $E^*$. We say that $E$ is  \emph{polyhedral} if
the unit ball of each of its finite-dimensional subspaces is a
polytope. Infinite-dimensional polyhedral Banach spaces were
introduced by Victor~Klee in \cite{Klee60}, and in the same paper it was proved 
that $c_0$ is polyhedral. 
Polyhedral Banach spaces were studied by various
authors, we refer the reader to \cite{hand} and the references
therein for the main basic results on the subject. 

The set  $\ext(B_{E^*})$ consisting of all extreme points of the dual unit ball $B_{E^*}$ plays a fundamental role in the study of polyhedral Banach spaces. Indeed, different notions of polyhedrality known in the literature are 
defined by means of geometric and topological properties of $\ext(B_{E^*})$  (for the different notions of polyhedrality see Definition~\ref{def:poli} below or \cite{InfPoly}). 
On the other hand, the results about polyhedrality and involving the set  $\ext(B_{E})$, consisting of all extreme points of the unit ball  $B_E$, are much more sporadic in the literature. This is due to the fact that extreme points of the unit ball of a polyhedral Banach space are in some sense 
 quite rare, and certain stronger notions of polyhedrality exclude existence of extreme points of $B_E$. For example, we have the following couple of facts:
\begin{itemize}
	\item  if 
$E$ is $(\mathrm{IV})$-polyhedral, then  $\ext(B_{E})=\emptyset$ (see \cite[Theorem~3.6]{InfPoly}) (a well-known example of this kind is $c_0$);
\item if $E$ is $(\mathrm{VI})$-polyhedral,  then each point in $\ext(B_E)$ is $\|\cdot\|$-isolated (see Observation~\ref{obs: extisolated}), and hence $\mathrm{card}\bigl(\ext(B_E)\bigr)\leq\mathrm{dens}(E)$.
 \end{itemize}

Existence of a   polyhedral Banach space  whose unit ball contains infinitely many extreme points was already observed in the famous paper 
\cite{Lindenstrauss},
 in which Joram~Lindenstrauss proved that
{\em every infinite-dimensional Banach space has a two-dimen\-sional
quotient whose unit ball is not a polygon, and hence no infinite-dimensional
dual Banach space is polyhedral}. 
For other examples of 
polyhedral Banach spaces  whose unit ball contains infinitely many extreme points we refer the reader to \cite{InfPoly}.

A natural question about $\ext(B_E)$, posed by Lindenstrauss himself in \cite{Lindenstrauss}, is whether
there exist a polyhedral
infinite-dimensional Banach space whose unit ball is the closed
convex hull of its extreme points. This problem was 
solved in the affirmative in the paper \cite{DEPOLY}, by considering a suitable renorming of $c_0$.
After that, other examples of   polyhedral Banach spaces with the same property were provided in \cite{Schreier}. We point out that  the examples contained both in \cite{DEPOLY} and in \cite{Schreier} are separable and satisfy   a stronger form of polyhedrality,  namely $(\mathrm{V})$-polyhedrality. 	The results contained in our paper concern the following natural question: to what extent, given a polyhedral Banach space,  is it possible to find a polyhedral renorming  whose unit ball  is the closed
convex hull of its extreme points?   
	Let us  briefly describe the structure and the main results of our paper.

After same notation and preliminaries, contained in Section~\ref{prelim},  in Section~\ref{sec: mainconstruction} we consider  a polyhedral Banach space $E$, satisfying a certain additional condition, called by us property $(**)$ and implied by $(\mathrm{V})$-polyhedrality. Then, given a positive number $\epsilon$, we define a certain $\epsilon$-equivalent  renorming of $E$, denoted by $X$.
The main ingredients used in our construction are:
\begin{enumerate}
	\item the well-known structural theorem by V.P.~Fonf, asserting that the unit sphere of an infinite-dimensional polyhedral Banach  space $E$ is covered by $\Gamma$'s many true faces (see Theorem~\ref{th: structure} below), where $\Gamma=\mathrm{dens}(E)$;
	\item Lemma~\ref{lemma: removefaces}, which, roughly speaking, asserts that  if we  do not remove  too much true faces from the unit sphere of $E$, then the unit ball $B_E$ coincides with the closed convex hull of the remaining true faces.
\end{enumerate} 
For the sake of clearness we divide the construction in two cases: the separable one and the non-separable one. Even if the two cases are formally almost identical, we think that the  idea of the construction can be more easily understood by the former.
Despite  some unavoidable technicalities, the  geometrical argument used in our construction is quite simple and  it is in part inspired by the results contained in \cite{DEPOLY} and in \cite[Section~5, (e)]{Lindenstrauss}.

In Section~\ref{sec: proofpoli}, we prove that $X$ is a polyhedral Banach space. In order to do that, we introduce a result, namely Proposition~\ref{prop: pyramidaldecomposition}, asserting that $B_X$ can be seen as the union of  $B_E$ and  $\Gamma$'s many sets, each of them obtained as the convex hull of a true face of $B_E$ and a singleton. Moreover, at the end of the section, we show that $B_X$  is the closed
convex hull of its extreme points. 

In Section~\ref{stronger}, we show that if in addition $E$ (the  Banach space taken as starting point of our construction)  is $(\mathrm{IV})$-polyhedral then  $X$   is $(\mathrm{V})$-polyhedral.
Finally, in Section~\ref{sec:final} we presents some easy consequences of our results and some open problems. For example, it remains open whether each polyhedral Banach space admits an equivalent polyhedral renorming such that 
the corresponding unit ball is the closed convex hull of its extreme points (cf. Problem~\ref{pb: without**}).

\section{Notation and preliminaries}\label{prelim}

Throughout all this paper, we consider real Banach spaces. If not differently stated, we assume that the Banach spaces considered are infinite-dimensional. Given a Banach space $E$, with topological dual $E^*$, we denote by $B_E$, $B_E^0$ and $S_E$ the closed unit ball, the open unit ball  and the unit sphere of $E$, respectively. 
Given an ordinal number $\Gamma$, it will be often identified  with the corresponding  interval of ordinal numbers $[0,\Gamma)$, endowed with the usual order. A set $B\subset[0,\Gamma)$ is called {\em cofinal} in $[0,\Gamma)$ if for each $\gamma\in [0,\Gamma)$ there exists
$b\in B$ such that $\gamma\leq b$. 

For $x,y\in E$, $[x,y]$ denotes the closed segment in $E$ with
endpoints $x$ and $y$, and $(x,y)=[x,y]\setminus\{x,y\}$ is the
corresponding ``open'' segment.
Even if the same notation may be used also for intervals in $\R$ and for intervals of ordinals, the  meaning should be clear from the context. We shall need the following elementary fact, the proof of which is left to the reader. 

\begin{fact}\label{fact: convex}
	Let $C,D$ be  subset of $E$ and suppose that, for every $x,y\in D$ with $x\neq y$, the segment $[x,y]$ intersects $C$. Then
	$$\conv (D)\subset \bigcup_{x\in D}\conv(C\cup\{x\}).$$
\end{fact}

In the sequel, if $B\subset E^*$,  we denote by $B'$ the set of all $w^*$-cluster points of $B$.
Let us recall that a set $\B\subset B_{E^*}$ is called 1-{\em norming} if for each $x\in E$ we have $\|x\|=\sup x(\B)$; using the Hahn-Banach theorem, it is easy to see that this is equivalent to say that $\cconv^{w^*}(\B)=B_{E^*}$. A set $\B\subset B_{E^*}$ is called {\em boundary} for $E$ if, for each $x\in E$ there exists $f\in \B$ such that $f(x)=\|x\|$.

Let us recall that the duality map $\D_E: S_E\to2^{S_{E^*}}$ is the function defined, for each $x\in S_E$, by
$$\D_E(x):=\{x^*\in S_{E^*};\, x^*x=1\}$$

\begin{definition} Let $K$ be a nonempty closed convex subset of $E$. 
	\begin{enumerate}
		\item An element  $x\in K$ is said to be an {\em extreme point} of $K$ if it does not lie in any ``open'' segment contained in $K$.
		\item By $\ext(K)$ we denote the set of all extreme points of $K$.
\item A {\em slice} of $K$ is a
set of the form
$$S(K,x^*,\alpha) = \{x\in K;\,  x^*(x)> \sup x(K)-\alpha\},$$
where $\alpha> 0$ and $x\in E^*\setminus\{0\}$ is bounded above on
$K$. 
		\item 
		An element  $x\in K$ is said  to be a {\em strongly exposed point} of $K$ if
		there exists  $x^*\in E^*\setminus\{0\}$ such that $x^*(x)=\sup x^*(K)$
		and for each norm neighbourhood $V$ of $x$ there exists $\alpha>0$ such that $S(K,x^*,\alpha)\subset
		V$. In this case, we say that $x$ is strongly exposed by $x^*$.
		\item We denote by $\mathrm{str\,exp}\,(K)$ the set of
		all strongly exposed points of $K$.
			\end{enumerate}
\end{definition}

\noindent We shall need the following  easy-to-prove fact. 

\begin{fact}\label{fact: exp} Let $C$ be a bounded subset of a normed space $E$. Let $x\in E$ and $f\in E^*$ be such that $f(x)>\sup f(C)$. Then $x$ is a strongly exposed point of the set $\cconv(C\cup\{x\})$ and it is strongly exposed by $f$.
\end{fact}

A closed convex and bounded set $P\subset E$ is called {\em polytope} if every finite-dimen\-sional section of $P$ is a finite-dimen\-sional polytope, i.e., the convex hull of finitely many points (equivalently, if every finite-dimen\-sional section of $P$ is a finite-dimen\-sional polyhedron, i.e., the intersection of finitely many halfspaces).
We say that the Banach space $E$ is  \emph{polyhedral} if $B_E$ is a polytope. It is well-known \cite{Klee59} that $E$ is polyhedral if{f} the  unit ball of each of its 2-dimensional subspaces is a
polytope. In \cite{InfPoly} and \cite{durpap}, the authors classified the principal known notions of infinite-dimensional polyhedrality, let us recall the main definitions.

\begin{definition}[{\cite[Definition~1.1]{InfPoly}}]\label{def:poli}
	Let $E$ be a Banach space and let us consider the following properties of $E$.
	\begin{enumerate}
		\item[(I)] $(\ext B_{E^*})'\subset\{0\}$;
		\item[(II)] $(\ext B_{E^*})'\subset r B_{E^*}$ for some $0<r<1$;
		\item[(III)] $(\ext B_{E^*})'\subset B_{E^*}^0$;
		\item[(IV)] $f(x)<1$ whenever $x\in S_E$ and $f\in(\ext B_{E^*})'$;
		\item[(V)] $\sup\{f(x); f\in\ext B_{E^*}\setminus \D_E(x)\}<1$ for each $x\in S_E$;
		\item[(VI)] every $x\in S_E$ has a neighborhood $V$ such that, for each $y\in V\cap S_E$, the segment $[x,y]$ lies in $S_E$;
		\item[(VII)] the set $M_v:=\{x\in S_E: \max v\bigl(\D_E(x)\bigr)\leq0\}$ is open in $S_E$ for each direction $v\in S_E$;
		\item[(K)] $E$ is polyhedral.
	\end{enumerate}
	For $j\in\{\mathrm{I,II,III,IV,V,VI,VII}\}$, $E$ is called $(j)$-polyhedral if it satisfies  property $(j)$. Moreover, $E$  is said isomorphically $\mathrm{(j)}$-polyhedral if it admits an equivalent norm satisfying property $(j)$.
\end{definition}

\begin{definition}\label{def: delta}
	We  say that $E$ satisfies property $(\Delta)$ if, for each $x\in S_X$, the set $\ext\bigl(\D_E(x)\bigr)$ is finite.	
\end{definition}

It is well known (see \cite{InfPoly} and the references therein) that
$$\mathrm{(I)}\Rightarrow \mathrm{(II)}\Rightarrow\mathrm{(III)}\Rightarrow\mathrm{(IV)}\Rightarrow\mathrm{(V)}\Rightarrow\mathrm{(VI)}\Rightarrow\mathrm{(VII)}\Rightarrow
\mathrm{(K)}$$
and that none of the implications above can be reversed. Moreover, $\mathrm{(IV)}$ implies $(\Delta)$.

\begin{definition}\label{def:5poli} Let $E$ be a Banach space and $\B\subset B_{E^*}$. 
	We say that $\B$ satisfies property $\mathrm{(IV)}$ if{f}
	we have, 
	\begin{equation}\label{eq:(iv)} \sup \{f(x);\ f\in\B'\}<1,\ \ \  \hbox{whenever}\ \ x\in S_E.
	\end{equation}	
	We say that $\B$ satisfies property $\mathrm{(V)}$ if{f}
	we have
	\begin{equation}\label{eq:(v)} \sup \{f(x);\ f\in\B\setminus \D_E(x)\}<1,\ \ \  \hbox{whenever}\ \ x\in S_E.
	\end{equation}
	Clearly, $E$ is $\mathrm{(IV)}$-polyhedral [respectively, $\mathrm{(V)}$-polyhedral] if{f} the set
	$\ext B_{E^*}$ satisfies property $\mathrm{(IV)}$ [respectively, property $\mathrm{(V)}$].
	
\end{definition}

\begin{remark}\label{remark:equiv(v)}
	Proceeding as in \cite[Section~2.4]{InfPoly}, it is immediate to prove that condition (\ref{eq:(v)})
	above is equivalent to the following condition:
	$$\bigl(\B\setminus \D_E(x)\bigr)'\cap \D_E(x)=\emptyset\ \ \  \hbox{whenever}\ \ x\in S_E.$$
	
\end{remark}

We shall say that a set $F\subset S_E$ is a {\em true face} of $B_E$ if there exists $f\in S_{X^*}$ such that, if  $H=f^{-1}(1)$, then $F=H\cap B_E$ and  $\inte_H F$ (the relative topological interior of $F$ in $H$) is nonempty. The set $\inte_H F$ will be simply called the {\em interior of the true face $F$}.
 We have the following important structural result for polyhedral Banach spaces (see \cite{fonfstrutt, fonfstruttnew, veselystrutt}).

\begin{theorem}\label{th: structure}
	Let $E$ be a polyhedral Banach space. Then the sphere $S_E$ is covered by the true faces of $B_E$. Hence the set 
	$$\B_0=\{f\in S_{E^*}; f^{-1}(1)\cap B_E\ \text{is a true face of}\ B_E  \}$$
	is a boundary for $E$. Moreover, $B_{E^*}=\cconv\B_0$ and $\mathrm{card}(\B_0)=\mathrm{dens}(E)=\mathrm{dens}(E^*)$, where $\mathrm{dens}(E)$ and $\mathrm{dens}(E^*)$ denote the  density character w.r.t. the norm topology of $E$ and $E^*$, respectively.
\end{theorem}

Observe that the boundary $\B_0$ from Theorem~\ref{th: structure} is minimal, in the sense that each boundary of $E$ contains $\B_0$. In the sequel we shall simply say that $\B_0$ is the {\em minimal boundary} of $E$. Moreover, Theorem~\ref{th: structure} implies  that polyhedral spaces are in particular Asplund spaces. 
The following result holds.

\begin{proposition}[{\cite[Lemma~2]{veselystrutt}}]\label{prop: gateaux-frechet-truefaces} Let $E$ be a polyhedral Banach space and $x_0\in S_E$. Then the following assertions are equivalent.
	\begin{enumerate}
		\item $x_0$ is in  the interior of a true face of $B_E$;\item $x_0$ is a point of Fr\'echet differentiability of $B_E$;
		\item $x_0$ is a point of  G\^ateaux differentiability of $B_E$. 
	\end{enumerate}
	
\end{proposition}

In the sequel, we shall need the following consequence of Theorem~\ref{th: structure}. 

\begin{corollary}\label{cor: cardinalityboundary}
	Suppose that $D_0$ is a polytope with nonempty interior in a Banach space $E$ and suppose that there exist families $\{g_i\}_{i\in I}\subset E^*$ and $\{a_i\}_{i\in I}\subset \R$ such that:
	\begin{enumerate}
		\item $D_0=\bigcap_{i\in I}\{x\in E;\, g_i(x)\leq a_i\}$;
		\item for each $d\in\partial D_0$, there exists $i\in I$ such that $g_i(d)=a_i$.
	\end{enumerate}
Then $\mathrm{card}(I)\geq\mathrm{dens}(E)$. 
\end{corollary}

\begin{proof}
	Without any loss of generality, we can suppose that $0\in\inte D_0$ and that $a_i=1$, whenever $i\in I$. Without any loss of generality, we can also suppose that the norm of $E$ is polyhedral and that the closed unit ball $B_E$ coincides with the set $D_0\cap (-D_0)$. Let us denote by $\mathcal B_0$ the minimal boundary of $E$, we claim that $\mathcal B_0\subset \{\pm g_i\}_{i\in I}$. In order to prove the claim, let $f\in\mathcal B_0$, $H=f^{-1}(1)$, $S=H\cap B_E$, and $x_0\in \mathrm{int}_H S$. Since $S_E\subset \partial D_0\cup \partial(-D_0)$, we can suppose that $x_0\in\partial D_0$, the case in which $x_0\in\partial(-D_0)$ can be treated similarly. Let $i\in I$ be such that $g_i(x_0)=1=\sup g_i(D_0)$; since $B_E\subset D_0$, we have that $\sup g_i(B_E)=1=g_i(x_0)$. Hence $g_i=f$ and the claim is proved. By our claim and  Theorem~\ref{th: structure}, we have $\mathrm{card}(I)\geq\mathrm{card}(\mathcal B_0)=\mathrm{dens}(E)$.
\end{proof}

The following lemma is an analogue of \cite[Remark~1.4]{FonfLindVes}, stated here for both properties $\mathrm{(IV)}$ and $\mathrm{(V)}$.  For the proof concerning property $\mathrm{(V)}$ see also \cite[Lemma~3.4]{DEPOLY}.

\begin{lemma}\label{lemma:5poli1normanti} Let $E$ be a Banach space. Then the following conditions are equivalent.
	\begin{enumerate}
		\item $E$ is $\mathrm{(IV)}$-polyhedral [respectively, $\mathrm{(V)}$-polyhedral]. \item There exists a
		1-norming set $\B\subset B_{E^*}$ satisfying property $\mathrm{(IV)}$ [respectively, property $\mathrm{(V)}$].
		\item $E$ is polyhedral and the minimal boundary $\B_0$ satisfies property $\mathrm{(IV)}$ [respectively, property $\mathrm{(V)}$].
	\end{enumerate}
\end{lemma}

 In the sequel, we shall need the following notion of polyhedrality. 

\begin{definition}[Property~($**$)]\label{def: property**} Let $E$ be a polyhedral space and let $\mathcal B_0$ its minimal boundary. We shall say that $E$ satisfies property $(**)$ if 
	\begin{equation}\label{eq:(**)} \sup \{f(x);\ f\in\mathcal B_0\setminus \D_E(x)\}<1,\ \ \  \hbox{whenever}\ \ x\in \mathrm{FD}(S_E),
	\end{equation}
	where  $\mathrm{FD}(S_E)$ denotes the set of all elements in $S_E$ that are Fr\'echet differentiability  points of $B_E$. 
\end{definition}

\section{The main construction}\label{sec: mainconstruction}

In this section, we suppose that $\epsilon>0$ and  that $E$ is an infinite-dimensional
polyhedral Banach space satisfying property~($**$). The norm of $E$ will be denoted by $\|\cdot\|$. For the sake of clearness we divide our construction in several subsections.

\subsection{Auxiliary results}

By Theorem~\ref{th: structure}, the unit sphere $S_E$ is covered by the true faces of $B_E$, let $\B_0$ be the corresponding minimal boundary. Let $\Gamma=\mathrm{card}(\B_0)$ and let us represent $\Gamma$ as the interval of ordinals $[0,\Gamma)$. Clearly, $\B_0$ can be represented as the disjoint union of two families $$\{f^+_\gamma; {\gamma\in[0,\Gamma)} \},\qquad \{f^-_\gamma; {\gamma\in[0,\Gamma)} \},$$
such that $f^+_\gamma=-f^-_\gamma$, whenever $\gamma\in[0,\Gamma)$. 
For each $\gamma\in[0,\Gamma)$, let us denote $$H^+_\gamma=(f^+_\gamma)^{-1}(1),\quad S^+_\gamma=H^+_\gamma\cap B_E,\quad  H^-_\gamma=(f^-_\gamma)^{-1}(1),\quad S^-_\gamma=H^-_\gamma\cap B_E.$$
By our hypotheses and by Theorem~\ref{th: structure}, we have $\Gamma=\mathrm{dens}(E)=\mathrm{dens}(E^*)$.
\smallskip 

Roughly speaking, the next lemma asserts that, if we do not remove  too much true faces from the unit sphere $S_E$, then the unit ball $B_E$ coincides with the closed convex hull of the remaining true faces. 

\begin{lemma}\label{lemma: removefaces}
	Suppose that $\mathcal C\subset [0,\Gamma)$ is such that $\mathrm{card}(\mathcal C)<\Gamma$, then
	\begin{equation}\label{eq: removefaces}
		B_E=\cconv\left(\bigcup_{w\in [0,\Gamma)\setminus \mathcal{C}}[S_w^+\cup S_w^-]\right).
	\end{equation} 
	In particular,  for each $z\in B_E$ and each $\theta>0$, there exist $A$, a finite nonempty subset of $[0,\Gamma)\setminus\mathcal C$, and  $G$, a finite nonempty set of   
	$$ \bigcup_{w\in A}[\mathrm{int}_{H^+_w}(S^+_w)\cup\mathrm{int}_{H^-_w}(S^-_w)],$$
	such that  $z\in\conv(G)+\theta B_E$ and such that, for each $w\in A$, the sets $G\cap \mathrm{int}_{H^+_w}(S^+_w)$ and $G\cap \mathrm{int}_{H^-_w}(S^-_w)$ are singletons.
\end{lemma}

\begin{proof} Let us prove \eqref{eq: removefaces}.
Suppose on the contrary that  $$\D:=\cconv\left(\bigcup_{w\in [0,\Gamma)\setminus \mathcal{C}}[S_w^+\cup S_w^-]\right)\neq B_E,$$
	then, by the Hahn-Banach theorem, there exists $g\in X^*$ such that $$\sup g(D)<1<\sup g(B_E).$$ Let us consider the set $D_0=\{x\in B_E;\, g(x)\geq1\}$. Then $D_0$ is a closed convex set with nonempty interior. Moreover, $D_0$ is clearly a polyhedron since it is the intersection between $B_E$, which is a polyhedron, and a closed half-space. We claim that $$
	D_0=F:=\bigcap_{w\in\mathcal C} \{x\in E;\, f^+_w(x)\leq1,\, f^-_w(x)\leq1\}\cap \{x\in E;\, g(x)\geq1\},$$
	and that, for each $d\in\partial D_0$, at least one of the following condition holds:
	\begin{itemize}
		\item $g(d)=1$;
		\item there exists $w\in\mathcal C$ such that $f^+_w(d)=1$ or $f^-_w(d)=1$.
	\end{itemize}
	\begin{proof}[Proof of the claim] 
The containment $D_0\subset F$ is clear. For the other containment it is clearly sufficient to prove that $F\subset B_E$. Suppose on the contrary that there exists $x\in F\setminus B_E$. Take any $x_0\in\inte D_0$ and observe that $g(x_0)>1$ and $f^+_w(x_0)<1, f^-_w(x_0)<1$, whenever $w\in\mathcal C$. Let $y\in S_E$ be such that $S_E\cap[x_0,x]=\{y\}$. Then clearly \begin{enumerate}
	\item $f^+_w(y)<1, f^-_w(y)<1$, whenever $w\in\mathcal C$;
	\item $g(y)>1$.
\end{enumerate}		Now, since $y\in S_E$, there exists $w\in[0,\Gamma)$ such that $f^+_w(y)=1$ or $f^-_w(y)=1$, and necessarily by (i) we have that $w\not\in \mathcal{C}$. Hence,  $y\in D$ and, by  the fact that $\sup g(D)<1$, we get a contradiction by (ii). The proof of $D_0=F$ is concluded.

For the second part of the claim, suppose that $d\in\partial D_0$ and that $g(d)>1$, then clearly $d\in S_E$ and hence there exists $w\in[0,\Gamma)$ such that $f^+_w(d)=1$ or $f^-_w(d)=1$. By the definition of $D$ and since $\sup g(D)<1$, we have that $w\in\mathcal C$. The proof of the claim is concluded.  
	\end{proof}
	Now, by our claim and Corollary~\ref{cor: cardinalityboundary}, we have that $\mathrm{card}{(\mathcal C)}\geq \mathrm{dens}(E)=\Gamma$, a contradiction and \eqref{eq: removefaces} is proved. 
	
	For the latter part, since 
	\begin{eqnarray*}
		\textstyle z\in B_E &\subset& \textstyle \conv\left(\bigcup_{w\in [0,\Gamma)\setminus \mathcal{C}}[S_w^+\cup S_w^-]\right)+\frac\theta2 B_E\\
		&\subset& \textstyle 
		\conv\left(\bigcup_{w\in [0,\Gamma)\setminus \mathcal{C}}[\mathrm{int}_{H^+_w}(S^+_w)\cup\mathrm{int}_{H^-_w}(S^-_w)]\right)+\theta B_E, 
	\end{eqnarray*}
	there exists  $A$, a finite nonempty subset of $[0,\Gamma)\setminus\mathcal C$, and  $G$, a finite nonempty set of   
	$$ \bigcup_{w\in A}[\mathrm{int}_{H^+_w}(S^+_w)\cup\mathrm{int}_{H^-_w}(S^-_w)],$$
	such that  $z\in\conv(G)+\theta B_E$. Since, for each $w\in A$, the sets $\mathrm{int}_{H^+_w}(S^+_w)$ and $\mathrm{int}_{H^-_w}(S^-_w)$ are convex, we can suppose without any loss of generality that, for each $w\in A$, the sets $G\cap \mathrm{int}_{H^+_w}(S^+_w)$ and $G\cap \mathrm{int}_{H^-_w}(S^-_w)$ contains at most an element. If necessary, we can add finitely many elements to the set $G$ in such a way that the thesis holds. 
\end{proof}

\noindent We are grateful to L.~Vesely for the idea of the proof of the present form of Lemma~\ref{lemma: removefaces}, indeed the first version of the proof of this lemma required the additional hypothesis $\mathrm{dens}(E)=w^*\text{-}\mathrm{dens}(E^*)$.
\smallskip

\noindent In the sequel we shall also need the following fact, for the sake of completeness we include a sketch of a proof.

\begin{fact}\label{lemma:theta}
	Let $\epsilon$ and $\Gamma$ be as above. Then there exists a function $\theta:[0,\Gamma)\to(0,\epsilon)$ such that 
	\begin{equation}\label{eq:proprtheta}
		\inf_{\gamma\geq \eta}\theta(\gamma)=0,\qquad \text{whenever}\  \eta\in[0,\Gamma).
	\end{equation}
\end{fact}
\begin{proof}[Sketch of the proof]
	Let us denote by $[\Gamma]^n$  the set of all $n$-elements subsets of $\Gamma$, and by $[\Gamma]^{<\infty}$  the set of all finite subsets of $\Gamma$.
	Observe that, whenever $n\in\N$, we have
	\begin{equation}\label{eq: insiemifiniti}
		\mathrm{card}\left([\Gamma]^n\right)=\mathrm{card}\left([\Gamma]^{<\infty}\right)=\Gamma.
	\end{equation}
	In particular, there exists a one-to-one correspondence $p:[0,\Gamma) \to \Gamma^{<\infty}$.   If $A\in \Gamma^{<\infty}$, let us denote by $\# A$ the number of its elements. By \eqref{eq: insiemifiniti}, it is easy to see that the map  $\theta:[0,\Gamma)\to(0,\epsilon)$, defined by 
	$$\theta(\gamma)=\frac{\epsilon}{2+\#p(\gamma)},\qquad\qquad\gamma\in[0,\Gamma),$$
	satisfies \eqref{eq:proprtheta}.
\end{proof}

 \subsection{Construction of the sets $A_\gamma$ and $E_\gamma$ in the separable case}\label{subsection: separableconstruction}
Let us suppose that $\Gamma=\omega$.
 Let $\{x_n\}_{n\in[0,\omega)}\subset B_E$ be dense in $B_E$ and let $\theta:[0,\omega)\to(0,\epsilon)$ be as in Fact~\ref{lemma:theta}.
Put $E_0=A_0=\emptyset$ and $H_0=1$. We are going to define inductively, w.r.t. $n\in(0,\omega)$, nonempty sets $E_n\subset B_E$, $A_n\subset[0,\omega)$, and positive integers $H_n$  such that:

\begin{enumerate}
	\item[{($P_1^n$)}] $A_n\cap A_m=\emptyset$, whenever ${m<n}$;
	\item[{($P_2^n$)}] $A_n$ is finite;
		\item[{($P_3^n$)}]  if we denote 
	$$W^n=\{x_m\}_{m<n}\cup \bigr(\bigcup_{m<n} E_m\bigl),$$
	then $W^n$ is finite and  we can write $W^n=\{z^n_k\}_{k< H_n}$;
	\item[{($P_4^n$)}] 
	there exist $\{A^n_k\}_{k<H_n}$, a finite family  of finite pairwise disjoint subsets of $[0,\omega)$, and $\{G^n_k\}_{k< H_n}$,  a finite family of finite subsets of $B_E$,   such that  $E_n=\bigcup_{k< H_n} G_k^n$, $A_n=\bigcup_{k< H_n}A_k^n$, and such that, for each $k< H_n$, the following conditions are satisfied:
	\begin{enumerate}[(a)]
			\item $\textstyle A^n_k\subset  [0,\omega)\setminus B^n_k$, where $B^n_k=\bigl(\bigcup_{m<n}A_m\bigr)\cup\bigl(\bigcup_{h<k}A_h^n\bigr)$;
				\item $G_k^n$ is a  subset of $$ \bigcup_{w\in A^n_k}[\mathrm{int}_{H^+_w}(S^+_w)\cup\mathrm{int}_{H^-_w}(S^-_w)];$$
		\item  we have  $z^n_k\in[\conv(G^n_k)+\theta(n)B_E]$ and moreover, for each $w\in A^n_k$, the sets $G^n_k\cap \mathrm{int}_{H^+_w}(S^+_w)$ and $G^n_k\cap \mathrm{int}_{H^-_w}(S^-_w)$ are singletons. 
	\end{enumerate} 
\end{enumerate}

Let us show that this is possible. Let $n\in(0,\omega)$ and suppose that $E_m$, $A_m$, and $H_m$, satisfying properties {($P_1^m$)}-{($P_4^m$)}, are already defined,  whenever $m<n$. 
Observe that, for $m<n$, the sets $E_m$ and $A_m$ are finite and hence also the set $\bigcup_{m<n}E_m$ is finite. Therefore we can find $H_n$ and $\{z^n_k\}_{k< H_n}$ such that  ($P_3^n$) is satisfied. 
  Now, observe that, by Lemma~\ref{lemma: removefaces}, 
  there exist $A^n_0$, a finite nonempty subset of $[0,\omega)\setminus\bigcup_{m<n}A_m$, and  $G^n_0$, a finite nonempty set of 
$$ \bigcup_{w\in A^n_0}[\mathrm{int}_{H^+_w}(S^+_w)\cup\mathrm{int}_{H^-_w}(S^-_w)],$$
such that  $z^n_0\in[\conv(G^n_0)+\theta(n)B_E]$ and such that, for each $w\in A^n_0$, the sets $G^n_0\cap \mathrm{int}_{H^+_w}(S^+_w)$ and $G^n_0\cap \mathrm{int}_{H^-_w}(S^-_w)$ are singletons.

Applying Lemma~\ref{lemma: removefaces} finitely many times, we can easily define, for $k< H_n$,  finite  sets $G^n_k$ and $A^n_k$ such that (a), (b), and (c) in ($P_4^n$) are satisfied.
Then define $E_n=\bigcup_{k< H_n} G_k^n$, $A_n=\bigcup_{k< H_n}A_k^n$. Clearly ($P_2^n$) holds, properties ($P_1^n$) and ($P_4^n$) follow by our construction.  
\medskip

\subsection{Construction of the sets $A_\gamma$ and $E_\gamma$ in the non-separable case}\label{subsection: nonseparableconstruction}

Let us suppose that $\Gamma>\omega$. Let $\{x_\gamma\}_{\gamma\in[0,\Gamma)}\subset B_E$ be dense in $B_E$ and let $\theta:[0,\Gamma)\to(0,\epsilon)$ be as in Fact~\ref{lemma:theta}. Put $A_\omega=\emptyset$ and $E_\omega=\emptyset$. We are going to define by transfinite induction, w.r.t. $\gamma\in(\omega,\Gamma)$, nonempty sets $E_\gamma\subset B_E$ and $A_\gamma\subset[0,\Gamma)$ such that:

\begin{enumerate}
	\item[{($P_1^\gamma$)}] $A_\gamma\cap A_\eta=\emptyset$, whenever ${\eta \in[\omega,\gamma)}$;
	\item[{($P_2^\gamma$)}] $\mathrm{card}\bigl(A_\gamma\bigr)\leq\mathrm{card}(\gamma)$;	\item[{($P_3^\gamma$)}]  if we denote 
	$$W^\gamma=\{x_\eta\}_{\eta\in[0,\gamma)}\cup \bigr(\bigcup_{\eta\in[\omega,\gamma)} E_\eta\bigl),$$
	we  have that $\mathrm{card}(W^\gamma)= \mathrm{card}(\gamma)$, and hence we can write $W^\gamma=\{z^\gamma_\eta\}_{\eta\in[0,\gamma)}$;
	\item[{($P_4^\gamma$)}] 
	there exist $\{A^\gamma_\eta\}_{\eta\in[0,\gamma)}$, a family  of finite parwise disjoint subsets of $[0,\Gamma)$, and $\{G^\gamma_\eta\}_{\eta\in[0,\gamma)}$,  a family of finite subsets of $B_E$,   such that  $E_\gamma=\bigcup_{\eta\in[0,\gamma)} G_\eta^\gamma$, $A_\gamma=\bigcup_{\eta\in[0,\gamma)}A_\eta^\gamma$, and such that, for each $\eta\in[0,\gamma)$, the following conditions are satisfied:
	\begin{enumerate}[(a)]
		\item $\textstyle A^\gamma_\eta\subset  [0,\Gamma)\setminus B^\gamma_\eta$, where $B^\gamma_\eta=\bigl(\bigcup_{v\in[\omega,\gamma)}A_v\bigr)\cup\bigl(\bigcup_{v\in[0,\eta)}A_v^\gamma\bigr)$;
		\item $G_\eta^\gamma$ is a  subset of $$ \bigcup_{w\in A^\gamma_\eta}[\mathrm{int}_{H^+_w}(S^+_w)\cup\mathrm{int}_{H^-_w}(S^-_w)];$$
		\item  we have  $z^\gamma_\eta\in[\conv(G^\gamma_\eta)+\theta(\gamma)B_E]$ and moreover, for each $w\in A^\gamma_\eta$, the sets $G^\gamma_\eta\cap \mathrm{int}_{H^+_w}(S^+_w)$ and $G^\gamma_\eta\cap \mathrm{int}_{H^-_w}(S^-_w)$ are singletons. 
	\end{enumerate} 
\end{enumerate}

Let us show that this is possible. Let $\gamma\in(\omega,\Gamma)$ and suppose that $E_\delta$ and $A_\delta$, satisfying properties {($P_1^\delta$)}-{($P_4^\delta$)}, are already defined,  whenever $\delta\in[\omega,\gamma)$. 
Observe that, for $\delta\in(\omega,\gamma)$, we have that:
\begin{itemize}
	\item ($P_2^\delta$) implies  that $\mathrm{card}\bigl(A_\delta\bigr)\leq\mathrm{card}(\delta)\leq \mathrm{card}(\gamma)$;
	\item ($P_4^\delta$)  implies  that  
	$\mathrm{card}\bigl(E_\delta\bigr)=\mathrm{card}\bigl(A_\delta\bigr)$.	
\end{itemize}  
Hence, $\mathrm{card}\bigl(\bigcup_{\eta\in[\omega,\gamma)}E_\eta\bigr)\leq\mathrm{card}(\gamma)<\Gamma$ and hence ($P_3^\gamma$) is satisfied. 
Now, observe that, by Lemma~\ref{lemma: removefaces}, 
there exist $A^\gamma_0$, a finite nonempty subset of $[0,\Gamma)\setminus\bigcup_{\eta\in[\omega,\gamma)}A_\eta$, and  $G^\gamma_0$, a finite nonempty set of 
$$ \bigcup_{w\in A^\gamma_0}[\mathrm{int}_{H^+_w}(S^+_w)\cup\mathrm{int}_{H^-_w}(S^-_w)],$$
such that  $z^\gamma_0\in[\conv(G^\gamma_0)+\theta(\gamma)B_E]$ and such that, for each $w\in A^\gamma_0$, the sets $G^\gamma_0\cap \mathrm{int}_{H^+_w}(S^+_w)$ and $G^\gamma_0\cap \mathrm{int}_{H^-_w}(S^-_w)$ are singletons.

We can define by transfinite induction, w.r.t. $\eta\in[0,\gamma)$,  finite  sets $G^\gamma_\eta$ and $A^\gamma_\eta$ such that (a), (b), and (c) in ($P_4^\gamma$) are satisfied.
Indeed, let $\eta\in(0,\gamma)$ and suppose that $G^\gamma_\delta$ and $A^\gamma_\delta$ are already defined,  whenever $\delta\in[0,\eta)$. Since, for each $\delta<\eta$, the set  $A^\gamma_\delta$ is finite, we have that $\mathrm{card}(B^\gamma_\eta)\leq\mathrm{card}(\gamma)<\Gamma$. Hence, by Lemma~\ref{lemma: removefaces} there exist $A^\gamma_\eta$, a finite nonempty subset of $[0,\Gamma)\setminus B^\gamma_\eta$, and  $G^\gamma_\eta$, a finite nonempty subset of $B_E$ 
such that (a), (b) and (c) are satisfied. 
This conclude the construction of the sets $G^\gamma_\eta$ and $A^\gamma_\eta$, for $\eta\in[0,\gamma)$.

Then define $E_\gamma=\bigcup_{\eta\in[0,\gamma)} G_\eta^\gamma$, $A_\gamma=\bigcup_{v\in[0,\gamma)}A_v^\gamma$. Since $A_v^\gamma$ is finite, whenever $v\in[0,\gamma)$, we have that $\mathrm{card}(A_\gamma)\leq\mathrm{card}(\gamma)$ and hence  ($P_2^\gamma$) holds.  Properties ($P_1^\gamma$) and ($P_4^\gamma$) follow by our construction.  
\medskip

\subsection{Construction of the equivalent norm $|\cdot|$}

We are going to define an equivalent norm $|\cdot|$ on $E$, by means of the sets $A_\gamma$ and $E_\gamma$ defined in the previous subsections.

{\em  We will describe the details of the construction of $|\cdot|$ and its main properties (namely, conditions (A)-(E)  and Proposition~\ref{prop: eqnorm} below)  only in the non-separable case, the corresponding definition and results in the separable case are formally identical.
	In the remaining part of this section,
	we assume that $E$ is a non-separable polyhedral Banach space satisfying condition $(**)$.
}

 By Proposition~\ref{prop: gateaux-frechet-truefaces}, it is easy to see that property $(**)$ is equivalent to the following condition:
 \begin{equation}\label{eq: mu}
 	 \sup_{\xi\in[0,\Gamma)\setminus\{\gamma\}}|f^+_\xi(x)|<1,\qquad\text{whenever}\ \gamma\in\Gamma\  \text{and}\ x\in \mathrm{int}_{H^+_\gamma}(S^+_\gamma).
 \end{equation}
 
\smallskip

\smallskip 

Now, for  $\gamma\in(\omega,\Gamma)$ and $\eta<\gamma$, let the sets $W^\gamma$, $A^\gamma_\eta$ and $G^\gamma_\eta$ be defined as in Subsection~\ref{subsection: nonseparableconstruction}. For $w\in A^\gamma_\eta$, let us denote $\{e^{+}_w\}=G^\gamma_\eta\cap \mathrm{int}_{H^+_w}(S^+_w)$ and $\{e^{-}_w\}=G^\gamma_\eta\cap \mathrm{int}_{H^-_w}(S^-_w)$.  By our construction, there exist,   nonnegative numbers $\lambda^{+}_w, \lambda^{-}_w$ ($w\in A^\gamma_\eta$), satisfying $\sum_{w\in A^\gamma_\eta}(\lambda^{+}_w +\lambda^{-}_w)=1$ and such that 
$$z_\eta^\gamma\in\left[\sum_{w\in A^\gamma_\eta}(\lambda^{+}_w e^{+}_w+\lambda^{-}_w e^{-}_w)\right]+\theta(\gamma)B_E.$$
Let us denote $A=\bigcup_{\gamma\in({\omega},\Gamma)}\bigcup_{\eta\in[0,\gamma)}A^\gamma_\eta$ and suppose that, for each $\gamma\in [0,\Gamma)$:
\begin{enumerate}
	\item $e_\gamma:=e^{+}_\gamma$, $s_\gamma:=e^{-}_\gamma$, $\alpha_\gamma:=\lambda^{+}_\gamma$ and 
	$\beta_\gamma:=\lambda^{-}_\gamma$,
	 whenever $\gamma\in A$ and $\lambda^{+}_\gamma \geq\lambda^{-}_\gamma$;
	 	\item $e_\gamma:=-e^{-}_\gamma$, $s_\gamma:=-e^{+}_\gamma$, $\alpha_\gamma:=\lambda^{-}_\gamma$ and 
	 $\beta_\gamma:=\lambda^{+}_\gamma$,
	 whenever $\gamma\in A$ and $\lambda^{+}_\gamma <\lambda^{-}_\gamma$
\item $e_\gamma$ is equal to an arbitrarily taken element of $\mathrm{int}_{H^+_\gamma}(S^+_\gamma)$, whenever $\gamma\not\in A$.
\end{enumerate}
 Observe that by definition we have $\sum_{w\in A^\gamma_\eta}\beta_w\leq\frac12$, whenever $\gamma\in(\omega,\Gamma)$ and $\eta\in[0,\gamma)$. Finally, define 
$Z=\bigcup_{\gamma\in[0,\Gamma)}\{e_\gamma,-e_\gamma\}$,
 and observe that, by our construction, the following conditions are satisfied.
\begin{enumerate}[(A)]
	\item $Z$ is a symmetric  subset of $$\bigcup_{\gamma\in[0,\Gamma) }[\mathrm{int}_{H^+_\gamma}(S^+_\gamma)\cup\mathrm{int}_{H^-_\gamma}(S^-_\gamma)].$$
	\item For each $\gamma<\Gamma$, we have 
	$$Z\cap\mathrm{int}_{H^-_\gamma}(S^-_\gamma)=\{e_\gamma\}\ \ \ \text{and}\ \ \  Z\cap\mathrm{int}_{H^-_\gamma}(S^-_\gamma)=\{-e_\gamma\}.$$
	  \item Since $E$ satisfies property $(**)$, by \eqref{eq: mu},   we have that 
	  $$\textstyle \mu_\gamma:=\sup_{\xi\in[0,\Gamma)\setminus\{\gamma\}}|f^+_\xi(e_\gamma)|<1,\qquad \text{whenever}\  \gamma\in[0,\Gamma).$$
	\item For each $\gamma\in[0,\Gamma)$, we can take  a  number  $\theta_\gamma\in(0,\epsilon)$ 
such that $$\theta_\gamma+\mu_\gamma+\theta_\gamma\mu_\gamma<1.$$
\end{enumerate}
Let us define   
$$B=\cconv\bigl(\{\pm(1+\theta_\gamma) e_\gamma;\, \gamma\in[0,\Gamma)\}\bigr),$$
and observe that $B$ is a symmetric closed convex set contained in $(1+\epsilon)B_E$. Moreover, observe that 
$$B\supset\cconv\bigl(\{\pm e_\gamma;\, \gamma\in[0,\Gamma)\}\bigr).$$

\begin{proposition}\label{prop: eqnorm}
	We have $B_E\subset B$.
\end{proposition}

\noindent For $\gamma\in[0,\Gamma)$, let the set $W^\gamma$ be defined as in the construction of the sets $A_\gamma$ and $E_\gamma$, contained in the Subsection~\ref{subsection: nonseparableconstruction}. In the proof of the proposition we shall need the following lemma.

\begin{lemma}\label{lemma: trick}
	Let $\eta,\gamma\in(\omega,\Gamma)$ be such that $\eta<\gamma$. If $y\in\conv(W^\eta)$ then there exist $\alpha,\beta\in[0,1]$ such that $\beta\leq\frac12$, $\alpha+\beta=1$ and
	$$y\in\alpha B+\beta\,\conv (W^\gamma)+\theta(\eta)B_E.$$
\end{lemma}

\begin{proof}
Since the sets $B$, $\conv (W^\gamma)$, and $B_E$ are convex, we can suppose, without any loss of generality, that $y\in W^\eta$.   
By ($P_4^{\eta}$), there exists $\overline\eta\in[0,\eta)$ and a finite set $A_{\overline\eta}^{\eta}\subset[0,\Gamma)$ such that 
$$y\in\left[\sum_{w\in A_{\overline\eta}^{\eta}}(\lambda^{+}_w e^{+}_w+\lambda^{-}_w e^{-}_w)\right]+\theta(\eta)B_E.$$
Then
$$y\in\left[\sum_{w\in A_{\overline\eta}^{\eta}}\alpha_w e_w+\sum_{w\in A_{\overline\eta}^{\eta}}\beta_w s_w\right]+\theta(\eta)B_E,$$
and hence, if we define $\alpha=\sum_{w\in A_{\overline\eta}^{\eta}}\alpha_w $ and $\beta=\sum_{w\in A_{\overline\eta}^{\eta}}\beta_w $, we have $$y\in\alpha B+\beta\,\conv (W^\gamma)+\theta(\eta)B_E.$$	
\end{proof}

\begin{proof}[Proof of Proposition~\ref{prop: eqnorm}]
It is clearly sufficient to prove that, for each $t>0$ and $\gamma_0\in[0,\Gamma)$, we have  	$x_{\gamma_0}\in B+tB_E$. Let $n\in\N$ be such that $\frac{n}{2^{n-1}}\leq t$ and take $\gamma_1,\ldots,\gamma_n\in(\omega,\Gamma)$ such that $\gamma_0<\gamma_1<\ldots<\gamma_n$ and such that $\theta(\gamma_i)\leq \frac{1}{2^{n-1}}$, whenever $i=1,\ldots,n$ (this is possible since \eqref{eq:proprtheta} holds). 
Since $x_{\gamma_0}\in\conv(W^{\gamma_1})$, by Lemma~\ref{lemma: trick}, there exist $\alpha_1,\beta_1\in[0,1]$ such that $\beta_1\leq\frac12$, $\alpha_1+\beta_1=1$, and
$$x_{\gamma_0}\in \alpha_1 B+\beta_1\conv(W^{\gamma_2})+\frac{1}{2^{n-1}}B_E.$$
Hence, there exist $b_1\in B$ and $w_1\in\conv(W^{\gamma_2})$ such that
$$x_{\gamma_0}\in \alpha_1 b_1+\beta_1w_1+\frac{1}{2^{n-1}}B_E.$$
Repeating $(n-1)$-times the same argument yields existence of elements $w_{n-1}\in\conv(W^{\gamma_n})$, $b_2,\ldots,b_{n-1}\in B$,  and numbers $\alpha_2,\beta_2,\ldots,\alpha_{n-1},\beta_{n-1}\in [0,1]$ satisfying $\beta_i\leq\frac12$ and $\alpha_i+\beta_i=1$, whenever $i=2,\ldots,(n-1)$, and such that, if we denote $$b:=\alpha_1b_1+\beta_1\alpha_2b_2+\ldots+\beta_{n-2}\ldots\beta_1\alpha_{n-1} b_{n-1},$$
then we have
$$\textstyle x_{\gamma_0}\in b+\beta_{n-1}\ldots\beta_1w_{n-1}+(1+\beta_1+\ldots+\beta_{n-2})\frac{1}{2^{n-1}}B_E.$$
In particular, since $b\in B$ (this is easy to see by the definition of $b$ and since  $0\in B$) and  since $W^{\gamma_n}\subset B_E$, we have
$$\textstyle x_{\gamma_0}\in B+\frac{1}{2^{n-1}}B_E+\frac{n-1}{2^{n-1}}B_E\subset B+\frac{n}{2^{n-1}}B_E\subset B+tB_E,$$
and the proof is concluded.
\end{proof}

	In particular, $B$ is the unit ball of an equivalent norm on $E$ and we have $$\textstyle B_E\subset B\subset (1+\epsilon)B_E.$$

From now on, we denote by $X$ the Banach space $E$ endowed by the equivalent norm $|\cdot|$ such that $B_X=B$.

\section{$X$ is a polyhedral Banach space satisfying $B_X=\cconv\bigl(\mathrm{ext}(B_X)\bigr)$}\label{sec: proofpoli}

Let $X$  be defined as in the previous section and let us denote by $\|\cdot\|$ the norm of $E$. The aim of this section is to prove that $X$ is a polyhedral Banach space. The main ingredient of our proof is  Proposition~\ref{prop: pyramidaldecomposition}, called by us {\em Pyramidal decomposition of $B_X$} and asserting that $B_X$ can be seen as the union of  $B_E$ and  $\Gamma$'s many sets (pyramids), each of them obtained as the convex hull of a true face of $B_E$ and a singleton (see Figure~\ref{FIG:pyramidal}). Moreover, at the end of the section, we show that $B_X$  is the closed
convex hull of its strongly exposed points (and hence of its extreme points). Let us start with the following lemma.

\begin{figure}
	\begin{center}
		\includegraphics[width=8cm]{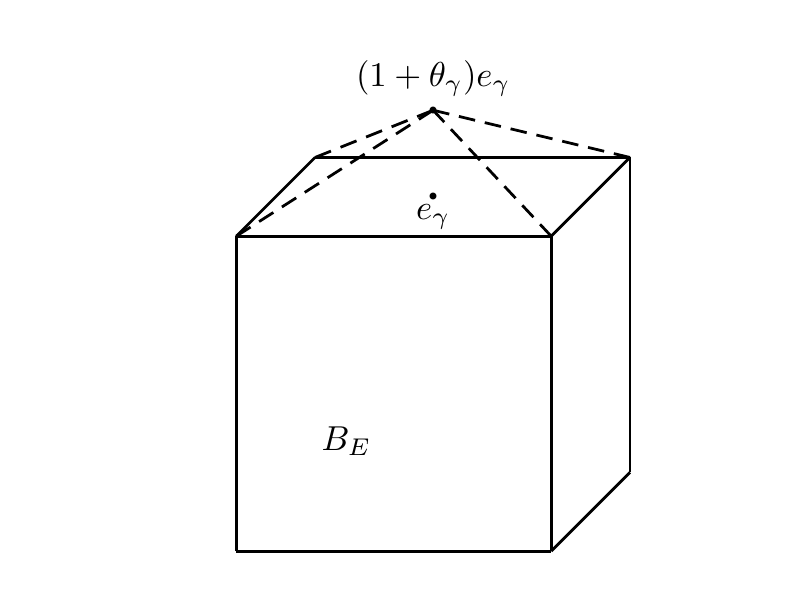}
	\end{center}
	\caption{The pyramid $\conv\bigl(B_E\cup\{\pm(1+\theta_\gamma) e_\gamma\})\bigr)$}\label{FIG:pyramidal}
\end{figure}

\begin{lemma}\label{lemma: solouno}
	We have 
	$$\textstyle 
	\conv\bigl(\{\pm(1+\theta_\gamma) e_\gamma;\, \gamma\in[0,\Gamma)\}\bigr)\subset\bigcup_{\gamma\in[0,\Gamma)}\conv\bigl(B_E\cup\{\pm(1+\theta_\gamma)e_\gamma\}\bigr).$$
\end{lemma}	
	
\begin{proof} By Fact~\ref{fact: convex} and symmetry, it
	is  sufficient to prove that, if $\gamma_1,\gamma_2,\in [0,\Gamma)$,  $\gamma_1\neq\gamma_2$,  and if we denote $w_i=(1+\theta_{\gamma_i})e_{\gamma_i}$ ($i=1,2$),  then   $[w_1,w_2]\cap B_E\neq\emptyset$. Suppose that this is not the case, then by the Hahn-Banach theorem there exists $f\in B_{E^*}$ such that $\inf f([w_1,w_2])>1$. Since $B_{E^*}=\cconv^{w^*}(\B_0)$, we can suppose without any loss of generality that $$f=\lambda_1f^+_{\gamma_1}+\lambda_2f^+_{\gamma_2}+\lambda_3 f_3\ldots+\lambda_nf_{n},$$
		where $n\in\N\setminus\{1,2\}$, $\lambda_1,\ldots,\lambda_n$ are non-negative numbers such that $\lambda_1+\ldots+\lambda_n=1$, and $f_3,\ldots,f_n\in\mathcal{B}_0\setminus\{f_{\gamma_1^+},f_{\gamma_2^+,}\}$. Since 
		$\mu_{\gamma_1}+\theta_{\gamma_1}+\mu_{\gamma_1}\theta_{\gamma_1}<1$, we have that 
		$\frac{1-\mu_{\gamma_1}-\mu_{\gamma_1}\theta_{\gamma_1}}{1+\theta_{\gamma_1}-\mu_{\gamma_1}-\mu_{\gamma_1}\theta_{\gamma_1}}>\frac12$. Since
		\begin{eqnarray*}
			\textstyle 1<f(w_1)&\leq&\lambda_1(1+\theta_{\gamma_1})+(\lambda_2+\ldots+\lambda_n)(1+\theta_{\gamma_1})\mu_{\gamma_1}\\
			&=&(1+\theta_{\gamma_1})\mu_{\gamma_1}+\lambda_1(1+\theta_{\gamma_1}-\mu_{\gamma_1}-\mu_{\gamma_1}\theta_{\gamma_1}),
		\end{eqnarray*}
		we have $\lambda_1>\frac{1-\mu_{\gamma_1}-\mu_{\gamma_1}\theta_{\gamma_1}}{1+\theta_{\gamma_1}-\mu_{\gamma_1}-\mu_{\gamma_1}\theta_{\gamma_1}}>\frac12$. A similar argument yields $\lambda_2>\frac12$, a contradiction and the proof is concluded.
		\end{proof}
  
\begin{proposition}[{Pyramidal decomposition of $B_X$}]\label{prop: pyramidaldecomposition}
We have 
\begin{eqnarray*}
  B_X&= &\bigcup_{\gamma\in[0,\Gamma)}\conv\bigl(B_E\cup\{\pm(1+\theta_\gamma) e_\gamma\})\bigr)\\
& =& B_E\cup \bigcup_{\gamma\in[0,\Gamma)}\conv\bigl(S^+_\gamma\cup\{(1+\theta_\gamma) e_\gamma\}\bigr)\cup \bigcup_{\gamma\in[0,\Gamma)}\conv\bigl(S^-_\gamma\cup\{-(1+\theta_\gamma) e_\gamma\}\bigr).	
\end{eqnarray*} 
\end{proposition}	

\begin{proof}
	Let us prove the first equality. The containment $\supset$ is trivial since $B_X=\cconv\bigl(\bigcup_{\gamma\in[0,\Gamma)}\{\pm(1+\theta_\gamma) e_\gamma\}\bigr)$ and since, by Proposition~\ref{prop: eqnorm}, $B_E\subset B_X$. For the other containment: let $x\in B_X$ and let $x_n\in \conv\bigl(\bigcup_{\gamma<\Gamma}\{\pm(1+\theta_\gamma) e_\gamma\}\bigr)$ ($n\in\N$) be such that $x_n\to x$ in the  norm-topology. By Lemma~\ref{lemma: solouno}, for each $n\in\N$, there exists $\gamma_n\in[0,\Gamma)$ such that $x_n\in\conv\bigl(B_E\cup\{\pm(1+\theta_{\gamma_n}) e_{\gamma_n}\}\bigr)$.
	Now, exactly one of the following conditions is satisfied.
	\begin{enumerate}
		\item The sequence $\{\gamma_n\}_n$ admits a constant subsequence.
		\item   For each $\gamma\in[0,\Gamma)$, we have that eventually (as $n\to\infty$) $f^+_\gamma(x_n)\leq 1$ and $f^-_\gamma(x_n)\leq 1$.
	\end{enumerate}
In the case (ii) is satisfied, we have that 
	 $f^+_\gamma(x)\leq 1$ and $f^-_\gamma(x)\leq 1$, whenever $\gamma\in[0,\Gamma)$. Since $\B_0$ is a boundary (and hence a 1-norming set) for $E$, we have that $x\in B_E$.
	In the case (i) is satisfied, there exists $\overline{\gamma}\in[0,\Gamma)$ such that
	$$\textstyle x\in\cconv\bigl(B_E\cup\{\pm(1+\theta_{\overline\gamma}) e_{\overline\gamma}\}\bigr)=\conv\bigl(B_E\cup\{\pm(1+\theta_{\overline\gamma}) e_{\overline\gamma}\}\bigr).$$
	 In both cases, we have $x\in \bigcup_{\gamma\in[0,\Gamma)}\conv\bigl(B_E\cup\{\pm(1+\theta_\gamma) e_\gamma\}\bigr)$.

	For the second equality, observe that
	the containment $\supset$ is trivial. For the other containment, it is sufficient to observe  that if $\gamma\in[0,\Gamma)$,  $x=(1+\theta_\gamma) e_\gamma$, and $b\in B_E$, then the segment $[b,x]$ intersects $H^+_\gamma$ in a unique point $x'$ (indeed, $f^+_\gamma(x)>1$ and $f^+_\gamma(b)\leq1$). Moreover, since $f^\pm_\xi(x)\leq1$ and $f^\pm_\xi(b)\leq1$, whenever $\xi\in[0,\Gamma)$ and $\xi\neq\gamma$, we have that $x'\in B_E\cap H^+_\gamma=S^+_\gamma$. The proof is concluded.     
\end{proof}

\begin{lemma}[{\cite[Lemma~2.7]{DEPOLY}}]\label{lemma:polytope+fin=polytope} Let $P$ be a polytope in a Banach space $Z$ and $W\subset Z$ a finite set. Then
	$K:=\conv(P\cup W)$ is a polytope.
\end{lemma}

\begin{theorem}\label{th: polyrenorming} $X$ is a polyhedral Banach space.
\end{theorem}

\begin{proof}
	Let  $Y$ be a finite-dimensional subspace of $X$. 
Since $B_E\cap Y$ is a polytope in $Y$ and hence it has finitely many true faces, the set 
		$$\Gamma_0:=\{\gamma\in[0,\Gamma);\, Y\cap\mathrm{int}_{H_\gamma^+}(S_\gamma^+)\neq\emptyset\},$$
			is finite.
	We claim that 
	$$\textstyle B_X\cap Y= Y\cap \conv\bigl(B_E\cup\{\pm(1+\theta_\gamma)e_\gamma;\, \gamma\in\Gamma_0\}\bigr).$$
	Suppose on the contrary that this is not the case, by Proposition~\ref{prop: pyramidaldecomposition} there exists $\gamma\in [0,\Gamma)\setminus\Gamma_0$, $b\in S_\gamma^+$, and $\lambda\in(0,1]$ such that $$x:=\lambda(1+\theta_\gamma)e_\gamma+(1-\lambda)b\in Y.$$
	Observe that
	$$\|x\|\geq f_\gamma^+(x)=\lambda(1+\theta_\gamma)+(1-\lambda)=1+\lambda\theta_\gamma.$$
	Now, if $\xi\in[0,\Gamma)\setminus\{\gamma\}$, we have 
		$$\textstyle f_\xi^\pm(\frac{x}{\|x\|})\leq\frac{\lambda(1+\theta_\gamma)\mu_\gamma+(1-\lambda)}{1+\lambda\theta_\gamma}\leq\frac{\lambda(\mu_\gamma+\theta_\gamma)+(1-\lambda)}{1+\lambda\theta_\gamma}<1.$$
		So, we have  	$$\textstyle \frac{x}{\|x\|}\in Y,\quad f_\gamma^+(\frac{x}{\|x\|})\leq1,\quad f_\gamma^-(\frac{x}{\|x\|})<0,\quad \sup\bigl\{f_\xi^\pm(\frac{x}{\|x\|});\,\xi\in[0,\Gamma)\setminus\{\gamma\}\bigr\}<1,$$
	hence $\frac{x}{\|x\|}\in Y\cap\mathrm{int}_{H_\gamma^+}(S_\gamma^+)$, a contradiction since $\gamma\in [0,\Gamma)\setminus\Gamma_0$ and the claim is proved.	
	By Lemma~\ref{lemma:polytope+fin=polytope} and our claim, the set $B_X\cap Y$  is a polytope in $Y$; the proof is concluded by the arbitrariness of $Y$.
\end{proof}

\begin{theorem}\label{th: extremepoints}
	$B_X=\cconv(\ext B_X)=\cconv\bigl(\mathrm{str\, exp} (B_X)\bigr)$.
\end{theorem}

\begin{proof}
	It is sufficient to prove that $$\{\pm(1+\theta_\gamma)e_\gamma;\, \gamma\in[0,\Gamma)\}\subset\mathrm{str\,exp} (B_X).$$
	Take $ \gamma\in[0,\Gamma)$ and let us prove that $(1+\theta_\gamma)e_\gamma$ is strongly exposed by $f_\gamma^+$. Let us start by observing that, by Proposition~\ref{prop: pyramidaldecomposition}, if $0<\alpha<\theta_\gamma$ then the slice $S(B_X,f_\gamma^+,\alpha)$ coincides with the slice $$S\bigl(\conv(B_E\cup\{(1+\theta_\gamma)e_\gamma\}),f_\gamma^+,\alpha\bigr).$$ 
	By Fact~\ref{fact: exp} and since $\sup f_\gamma^+(B_E)=1<f_\gamma^+\bigl((1+\theta_\gamma)e_\gamma\bigr)$, we have that $(1+\theta_\gamma)e_\gamma$ is a strongly exposed point of
	$\conv(B_E\cup\{(1+\theta_\gamma)e_\gamma\})$ and it is strongly exposed by $f_\gamma^+$. By our previous observation, $(1+\theta_\gamma)e_\gamma$ is a strongly exposed point of
	$B_X$.
\end{proof}

\section{Stronger polyhedrality properties: $X$ is $(\mathrm{V})$-polyhedral}\label{stronger}%

The aim of this section is to prove that, under the assumption that $E$ is  
$(\mathrm{IV})$-polyhedral, the Banach space $X$ is $(\mathrm{V})$-polyhedral. In the first part of the section we construct a boundary $\B$ of $X$, then in Theorem~\ref{th: Vpoly} we prove that  
if $E$ is  
$(\mathrm{IV})$-polyhedral then $\B$ satisfies property $(\mathrm{V})$.
\smallskip

Let $\gamma\in[0,\Gamma)$ and $\xi\in[0,\Gamma)\setminus\{\gamma\}$. Since $f^+_\gamma\bigl((1+\theta_\gamma)e_\gamma\bigr)=1+\theta_\gamma>1$ and  $$f^+_\xi\bigl((1+\theta_\gamma)e_\gamma\bigr)\leq (1+\theta_\gamma)\mu_\gamma\leq \theta_\gamma+\mu_\gamma<1,\quad f^-_\xi\bigl((1+\theta_\gamma)e_\gamma\bigr)<1$$ there exist (unique)  $\lambda^+_{\xi,\gamma},\lambda^-_{\xi,\gamma}\in(0,1)$ such that 
$$\textstyle \bigl[(1-\lambda^+_{\xi,\gamma})f^+_\xi+\lambda^+_{\xi,\gamma}f^+_\gamma\bigr]\bigl((1+\theta_\gamma)e_\gamma\bigr)=\bigl[(1-\lambda^-_{\xi,\gamma})f^-_\xi+\lambda^-_{\xi,\gamma}f^+_\gamma\bigr]\bigl((1+\theta_\gamma)e_\gamma\bigr)=1.$$
For $\gamma$ and $\xi$ as above, let us denote $$g^+_{\xi,\gamma}=(1-\lambda^+_{\xi,\gamma})f^+_\xi+\lambda^+_{\xi,\gamma}f^+_\gamma,\quad  g^-_{\xi,\gamma}=(1-\lambda^-_{\xi,\gamma})f^-_\xi+\lambda^-_{\xi,\gamma}f^+_\gamma,$$
then define
$$\mathcal{B}=\bigl\{\pm g^+_{\xi,\gamma}, \pm g^-_{\xi,\gamma};\, \gamma\in[0,\Gamma),\  \xi\in[0,\Gamma)\setminus\{\gamma\} \bigr\}.$$
Figure~\ref{FIG:pyramidalboundary2} describes the intuitive geometrical idea behind the definition of $\B$.

\begin{figure}
	\begin{center}
		\includegraphics[width=12cm]{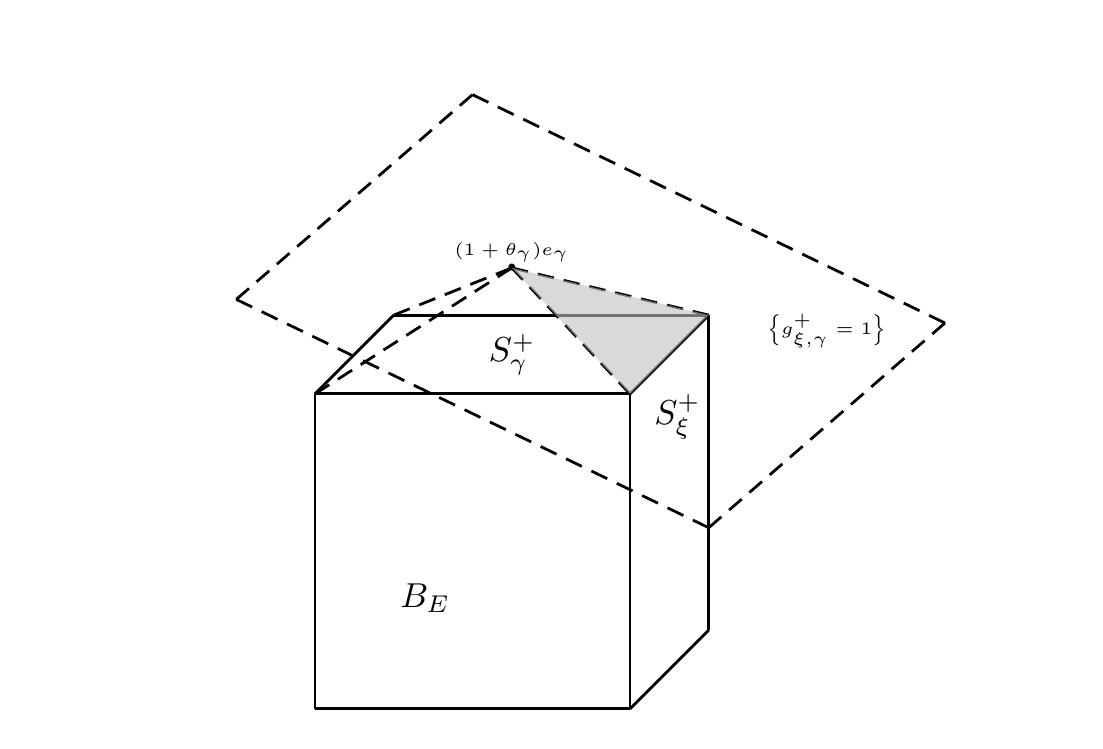}
	\end{center}
	\caption{The hyperplane given by $g^+_{\xi,\gamma}=1$ }\label{FIG:pyramidalboundary2}
\end{figure}

\begin{lemma}\label{lemma: inequalitylambda}
	Let $\gamma\in[0,\Gamma)$ and $\xi\in[0,\Gamma)\setminus\{\gamma\}$, then $\lambda^+_{\xi,\gamma}>\frac12$ and  $\lambda^-_{\xi,\gamma}>\frac12$.
\end{lemma}

\begin{proof} By definition of $g^+_{\xi,\gamma}$ and since  $\theta_\gamma+\mu_\gamma+\mu_\gamma\theta_\gamma<1$, we have
		\begin{eqnarray*}
		\textstyle 1=g^+_{\xi,\gamma}\bigl((1+\theta_\gamma)e_\gamma\bigr)&\leq&(1-\lambda^+_{\xi,\gamma})(1+\theta_\gamma)\mu_\gamma+\lambda^+_{\xi,\gamma}(1+\theta_\gamma)\\
		&=& \mu_\gamma+\mu_\gamma\theta_\gamma+\lambda^+_{\xi,\gamma}(1+\theta_\gamma-\mu_\gamma-\mu_\gamma\theta_\gamma)\\
		&<& \mu_\gamma+\mu_\gamma\theta_\gamma+2\lambda^+_{\xi,\gamma}(1-\mu_\gamma-\mu_\gamma\theta_\gamma).
	\end{eqnarray*}
Hence, $\lambda^+_{\xi,\gamma}>\frac12$ and similarly we have  $\lambda^-_{\xi,\gamma}>\frac12$. 
\end{proof}

In the sequel, we shall need the following fact, the proof of which is left to the reader.

\begin{fact}\label{fact: nonemptyinterior}
	Let $H$ be a hyperplane in a Banach space $X$, $w\in X\setminus H$, and let $S\subset H$ be a convex set such that $\mathrm{int}_H S\neq\emptyset$. Then the set $\conv(S\cup\{w\})$ has nonempty interior in $X$. More precisely, we have
	$$\textstyle \inte[\conv(S\cup\{w\})]=\bigcup_{\lambda\in(0,1)} [\lambda\,\mathrm{int}_H (S)+(1-\lambda) w].$$	
\end{fact}

\begin{proposition}\label{prop: boundary}
	$\mathcal{B}$ is contained in $B_{X^*}$ and it is a boundary for $X$.
\end{proposition}

\begin{proof} 
	Let $\gamma,\eta\in[0,\Gamma)$, and $\xi\in[0,\Gamma)\setminus\{\gamma\}$, then the following assertions hold. 
	\begin{enumerate}
		\item If $\eta=\gamma$, we have
		$$\textstyle g^+_{\xi,\gamma}\bigl((1+\theta_\eta)e_\eta\bigr)= 1=g^-_{\xi,\gamma}\bigl((1+\theta_\eta)e_\eta\bigr);$$
		\item If $\eta=\xi$, we have
		\begin{eqnarray*}
			\textstyle g^+_{\xi,\gamma}\bigl((1+\theta_\eta)e_\eta\bigr)&\leq&(1-\lambda^+_{\xi,\gamma})(1+\theta_\xi)+\lambda^+_{\xi,\gamma}(1+\theta_\xi)\mu_\xi\\
			&=& 1+\theta_\xi+\lambda^+_{\xi,\gamma}(\mu_\xi+\theta_\xi\mu_\xi-1-\theta_\xi)\leq1,
		\end{eqnarray*}	where the last inequality holds since, by Lemma~\ref{lemma: inequalitylambda}, $\lambda^+_{\xi,\gamma}>\frac12$,  and since by definition   $\mu_\xi+\theta_\xi\mu_\xi-1-\theta_\xi< -2\theta_\xi$. Similarly, we have $g^-_{\xi,\gamma}\bigl((1+\theta_\eta)e_\eta\bigr)\leq1$.
	\item If $\eta\neq\xi$ and $\eta\neq\gamma$, we have
	$$\textstyle 
		 g^+_{\xi,\gamma}\bigl((1+\theta_\eta)e_\eta\bigr)\leq(1-\lambda^+_{\xi,\gamma})(1+\theta_\eta)\mu_\eta+\lambda^+_{\xi,\gamma}(1+\theta_\eta)\mu_\eta<1,$$
		 and similarly $
		 g^-_{\xi,\gamma}\bigl((1+\theta_\eta)e_\eta\bigr)<1$.
	 \end{enumerate}
	Since $B_X=\cconv\bigl(\{\pm(1+\theta_\gamma) e_\gamma;\, \gamma\in[0,\Gamma)\}\bigr)$, we have $\sup_{x\in B_X}\sup x(\mathcal B)\leq 1$ and hence $\mathcal B\subset B_{X^*}$. 
	
	Let us prove that $\mathcal B$ is a boundary.
	 By Proposition~\ref{prop: pyramidaldecomposition}
	and Fact~\ref{fact: nonemptyinterior},  it is not difficult to see that 
	$S_X$ is contained in the union of the following two sets $$ \bigcup_{\gamma\in[0,\Gamma)}\bigcup_{\lambda\in[0,1]}\bigl[\lambda\partial_{H^+_\gamma}S^+_\gamma+(1-\lambda)(1+\theta_\gamma) e_\gamma\bigr],$$
	$$ \bigcup_{\gamma\in[0,\Gamma)}\bigcup_{\lambda\in[0,1]}\bigl[\lambda\partial_{H^-_\gamma}S^-_\gamma-(1-\lambda)(1+\theta_\gamma) e_\gamma\bigr].$$
		
Hence, if $x\in S_X$, we can suppose without any loss of generality, that there exists $\phi\in[0,1]$,	 $\gamma\in[0,\Gamma)$, $\xi\in[0,\Gamma)\setminus\{\gamma\}$, and $b\in (S_\gamma^+\cap S_\xi^+)\cup (S_\gamma^+\cap S_\xi^-)$ such that $x=(1-\phi)b+\phi (1+\theta_\gamma)e_\gamma$. 
	Hence,  $g^+_{\xi,\gamma}(x)=1$ or $g^-_{\xi,\gamma}(x)=1$. By the arbitrariness of $x$, the proof is concluded.
\end{proof}

\begin{theorem}\label{th: Vpoly}
Suppose that $E$ is $(\mathrm{IV})$-polyhedral, then $\mathcal B$ satisfies property $(V)$, and hence $X$ is a $(\mathrm{V})$-polyhedral Banach space.	
\end{theorem}

\begin{proof}
	Let us prove that, for each $x\in S_X$, we have $$ \sup \{f(x);\ f\in\B\setminus \D_X(x)\}<1.$$ By Remark~\ref{remark:equiv(v)}, this is equivalent to prove that,   for each $x\in S_X$, we have $$\bigl(\B\setminus \D_X(x)\bigr)'\cap \D_X(x)=\emptyset.$$
Fix $x\in S_X$, and observe that, proceeding as in the final part of the proof of Proposition~\ref{prop: boundary},  we can suppose without any loss of generality, that there exists $\phi\in[0,1]$,	 $\gamma\in[0,\Gamma)$, and $\xi\in[0,\Gamma)\setminus\{\gamma\}$ such that $x=(1-\phi)b+\phi (1+\theta_\gamma)e_\gamma$, where $b\in (S_\gamma^+\cap S_\xi^+)\cup (S_\gamma^+\cap S_\xi^-)$.
We consider three different cases.
\smallskip

\noindent {\bf Case  $\phi=0$.} In this case $x\in \partial_{H_\gamma^+} S_\gamma^+\subset S_E$. Suppose on the contrary that 
there exists a net $\{g_\alpha\}_{\alpha\in I}\subset \B\setminus \D_X(x)$ such that $g_\alpha {\to}g\in D_X(x)$ in the $w^*$-topology. Observe that since each element of $\B$ is a convex combination of two elements of $\B_0$, there exist $\mu_\alpha^1,\mu_\alpha^2\in[0,1]$ and $h_\alpha^1,h_\alpha^2\in\B_0$ such that $g_\alpha=
\mu_\alpha^1h_\alpha^1+\mu_\alpha^2h_\alpha^2$, whenever $\alpha\in I$. Moreover, by $w^*$-compactness and since $\{g_\alpha\}_{\alpha\in I}\subset \B\setminus \D_X(x)$,  we can suppose without any loss of generality that:
\begin{itemize}
	\item $h_\alpha^1\not\in D_E(x)$, whenever $\alpha\in I$ (indeed, if both $h_\alpha^1(x)=1$ and $h_\alpha^2(x)=1$, then $g_\alpha(x)=1$ and hence $g_\alpha\in D_X(x)$);
	\item $h_\alpha^1 {\to} h^1\in B_{E^*}$ and $h_\alpha^2 {\to} h^2\in B_{E^*}$ in the $w^*$-topology;
	\item $\mu_\alpha^1 {\to} \mu^1\in [0,1]$, $\mu_\alpha^2 {\to} \mu^2\in [0,1]$, and $\mu^1+\mu^2=1$.
\end{itemize}
 We claim that $\mu^1=0$ (and hence that $\mu^2=1$). Indeed, suppose that this is not case and observe that, since $1=g(x)=
 \mu^1h^1(x)+\mu^2h^2(x)$, we should have $1= h^1(x)$, a contradiction by the fact that 
$h_\alpha^1\not\in D_E(x)$, whenever $\alpha\in I$, and by the fact that $E$ is ($IV$)-polyhedral.
The same argument shows that, by passing to a subnet if necessary, we can suppose that 
$h_\alpha^2\in D_E(x)$, whenever $\alpha\in I$. Since $E$  is ($IV$)-polyhedral,  it has property $(\Delta)$ (see Definition~\ref{def: delta}), and hence  $D_E(x)\cap\mathcal B_0$ is finite; so, by passing to a subnet if necessary,  we can suppose that there exists $\overline\gamma\in[0,\Gamma)$ such that 
$h_\alpha^2=f^+_{\overline\gamma}$, whenever $\alpha\in I$. Since $\sup f^+_{\overline\gamma}(B_X)\geq f^+_{\overline\gamma}\bigl((1+\theta_{\overline{\gamma}})e_{\overline{\gamma}}\bigr)$ and since 
$\mu_\alpha^2 {\to} \mu^2=1$, we have that eventually $\sup g_\alpha(B_X)>1$, a contradiction since 
$\{g_\alpha\}_{\alpha\in I}\subset \B\subset B_{X^*}$.
\smallskip

\noindent {\bf Case  $\phi\in(0,1)$.} 
Let us observe that in this case we have:
\begin{itemize}
	\item    $g^+_{\xi,\gamma}(x)=1$ if and only if $f_\xi^+(b)=1$;
\item    $g^-_{\xi,\gamma}(x)=1$ if and only if $f_\xi^-(b)=1$;
	\item if $\eta\neq \gamma$ and $\xi\in\Gamma\setminus\{\eta\}$, then 
	\begin{eqnarray*}
	\textstyle g^+_{\xi,\eta}[(1+\theta_\gamma)e_\gamma]&\leq&(1-\lambda^+_{\xi,\eta})(1+\theta_\gamma)+\lambda^+_{\xi,\eta}(1+\theta_\gamma)\mu_\gamma\\
	&=& 1+\theta_\gamma+\lambda^+_{\xi,\eta}(\mu_\gamma+\theta_\gamma\mu_\gamma-1-\theta_\gamma)\\
	&\leq&\textstyle \frac{1+\theta_\gamma+\mu_\gamma+\theta_\gamma\mu_\gamma}{2},
\end{eqnarray*}	where the last inequality holds since, by Lemma~\ref{lemma: inequalitylambda}, $\lambda^+_{\xi,\eta}>\frac12$. Since by definition   $\mu_\gamma+\theta_\gamma\mu_\gamma+\theta_\gamma<1$, we have $g^+_{\xi,\eta}(x)<1$. Similarly $g^-_{\xi,\eta}(x)<1$.
\end{itemize} By the previous observation, we have  
$$\textstyle \B\cap D_X(x)=\{ g^+_{\xi,\gamma};\,   \xi\in\Gamma\setminus\{\gamma\}, f_\xi^+(b)=1 \}\cup\{ g^-_{\xi,\gamma};\,  \xi\in\Gamma\setminus\{\gamma\}, f_\xi^-(b)=1 \}.$$
Hence, we have that $\B\setminus D_X(x)=\B^1\cup B^2$, where
$$\B^1=\{ g^+_{\xi,\gamma};\,   \xi\in\Gamma\setminus\{\gamma\}, f_\xi^+(b)<1 \}\cup\{ g^-_{\xi,\gamma};\,  \xi\in\Gamma\setminus\{\gamma\}, f_\xi^-(b)<1 \},$$
$$\B^2= \bigl\{g^+_{\xi,\eta};\,   \eta\in\Gamma\setminus\{\gamma\},\,\xi\in\Gamma\setminus\{\eta\} \bigr\}\cup \bigl\{g^-_{\xi,\eta};\,  \eta\in\Gamma\setminus\{\gamma\},\,\xi\in\Gamma\setminus\{\eta\} \bigr\}.$$
Proceeding as in the observation at the beginning of this case we get
\begin{equation}\label{eq:supboundary}
	\sup_{g\in\B^2}g[(1+\theta_\gamma)e_\gamma]\leq\textstyle\frac{1+\theta_\gamma+\mu_\gamma+\theta_\gamma\mu_\gamma}{2}<1.
\end{equation}

Now, suppose on the contrary that 
there exists a net $\{g_\alpha\}_{\alpha\in I}\subset \B\setminus \D_X(x)$ such that $g_\alpha {\to}g\in D_X(x)$ in the $w^*$-topology.
By \eqref{eq:supboundary} and by passing to a subnet if necessary, we can suppose that 
$\{g_\alpha\}_{\alpha\in I}\subset \B^1$. Since $g(x)=1$, we necessarily have that $g(b)=1$. Moreover, since $\{g_\alpha\}_{\alpha\in I}\subset \B^1$, we have that $g_\alpha(b)<1$, whenever $\alpha\in I$. Proceeding as in the previous case, we get a contradiction.  
\smallskip

\noindent {\bf Case  $\phi=1$.} 
 In this case $x=(1+\theta_\gamma)e_\gamma$, let us observe that we have:
 \begin{itemize}
 	\item if $\xi\in\Gamma\setminus\{\gamma\}$, then  by definition $g^+_{\xi,\gamma}(x)=1$ and $g^-_{\xi,\gamma}(x)=1$;
 	\item if $\eta\neq \gamma$ and $\xi\in\Gamma\setminus\{\eta\}$, then proceeding as in the previous case, we have: $$\textstyle g^+_{\xi,\eta}[(1+\theta_\gamma)e_\gamma]\leq \frac{1+\theta_\gamma+\mu_\gamma+\theta_\gamma\mu_\gamma}{2},\quad g^-_{\xi,\eta}[(1+\theta_\gamma)e_\gamma]\leq \frac{1+\theta_\gamma+\mu_\gamma+\theta_\gamma\mu_\gamma}{2}.$$ 
 \end{itemize} By the previous observation, we have   that $$\B\setminus D_X(x)= \bigl\{g^+_{\xi,\eta};\,   \eta\in\Gamma\setminus\{\gamma\},\,\xi\in\Gamma\setminus\{\eta\} \bigr\}\cup \bigl\{g^-_{\xi,\eta};\,  \eta\in\Gamma\setminus\{\gamma\},\,\xi\in\Gamma\setminus\{\eta\} \bigr\}.,$$
and hence, proceeding as in the previous cases, we have
$$\textstyle  \sup \{f(x);\ f\in\B\setminus \D_X(x)\}\leq \frac{1+\theta_\gamma+\mu_\gamma+\theta_\gamma\mu_\gamma}{2}<1.$$
The proof is concluded.
\end{proof}

\section{Conclusion, final remarks, and open problems}\label{sec:final}

 Let us resume the results obtained in the previous sections.

\begin{definition}
	Let $E$ be a Banach space, $\epsilon>0$. The space $E$ endowed with an equivalent norm, with unit ball $B$, is called {\em$\epsilon$-equivalent renorming}  of $E$, if 
	$$B_E\subset B\subset (1+\epsilon)B_E.$$
\end{definition} 
\smallskip

\begin{theorem}	Let $E$ be a  polyhedral Banach space satisfying property $(**)$ and  $\epsilon>0$. Then  the Banach space  $X$, defined as in Section~\ref{sec: mainconstruction}, is an $\epsilon$-equivalent polyhedral  renorming of $E$ satisfying \begin{equation}\label{eq: generating}
		B_X=\cconv\bigl(\mathrm{ext}(B_X)\bigr).
	\end{equation} Moreover, if $E$ is $(\mathrm{IV})$-polyhedral then $X$ is  $(\mathrm{V})$-polyhedral.
\end{theorem}  

Whether the hypothesis concerning property ($**$) can be omitted in the first part of the previous result remains an open question.

\begin{problem}\label{pb: without**}
Suppose that $E$ is a  polyhedral Banach space. Does $E$ admits an equivalent polyhedral renorming $X$ satisfying $B_X=\cconv\bigl(\mathrm{ext}(B_X)\bigr)$? 
\end{problem}

Let us  point out that the problem above has an affirmative answer in the case $E$ is separable. Moreover, in this case, it is also possible to prove that the equivalent norms satisfying \eqref{eq: generating} are dense w.r.t. the Banach-Mazur distance (see Subsection~\ref{subsection: approximation} below for the details).
Problem~\ref{pb: without**} has an affirmative answer also in the case $E$ is a Lindenstrauss space, indeed in this case property $(**)$ is automatically satisfied (see Subsection~\ref{subsection: lindenstrauss}).

\subsection{Approximation by polyhedral norms such that the  corresponding closed unit ball  is the closed	convex hull of its extreme points}\label{subsection: approximation}

\noindent The following result holds (see {\cite[Theorem~3.3]{InfPoly}} and \cite[Theorem~1.1]{devillefonfhajek}).

\begin{theorem}\label{th: +rinormable}
	For a separable Banach space $E$ the following statements are equivalent.
	\begin{enumerate}
		\item For each $\epsilon>0$, $E$ admits an $\epsilon$-equivalent  $(\mathrm{IV})$-polyhedral renorming.
		\item $E$ is isomorphically polyhedral.
	\end{enumerate}
\end{theorem}

\noindent Whether the previous theorem holds in the non-separable case is not known. However, the results contained in  \cite{BibleSmith} imply that this is the case for the Banach space $c_0(\Gamma)$. By combining, the previous results we get the following corollary.

\begin{corollary}	Suppose that we are in one of the following cases:
	\begin{itemize}
		\item $E$ is isomorphic to $c_0(\Gamma)$ for some   nonempty set $\Gamma$;
		\item  $E$ is a separable isomorphically polyhedral Banach space.
	\end{itemize}
Then, for each $\epsilon>0$, $E$ admits an $\epsilon$-equivalent  $(\mathrm{V})$-polyhedral renorming such that the  corresponding closed unit ball  is the closed	convex hull of its extreme points.
\end{corollary}  

\subsection{Lindenstrauss spaces}\label{subsection: lindenstrauss}

A Banach space $E$ is called an {\em $L_1$-predual space} or a {\em Lindenstrauss space} if
its dual is isometric to $L_1(\mu)$ for some measure $\mu$. The following result about polyhedral Lindenstrauss spaces holds.

\begin{theorem}[{\cite[Theorem~4.3]{casinimiglierinapiaseckivesely}} ]
	 Let $E$ be a Lindenstrauss space. The following properties are
	equivalent:
	\begin{enumerate}
		\item $E$ is a polyhedral space; 
	\item 
	$E$ does not contain an isometric copy of $c$;
	\item $E$ is $(\mathrm{V})$-polyhedral.
\end{enumerate}
\end{theorem}

The theorem above implies that if $E$ is a polyhedral Lindenstrauss space then condition $(**)$ is satisfied, hence we have the following corollary.

\begin{corollary}
	Let $E$ be a polyhedral Lindenstrauss space. Then, for each $\epsilon>0$, $E$ admits an $\epsilon$-equivalent polyhedral renorming such that the  corresponding closed unit ball  is the closed	convex hull of its extreme points.
\end{corollary}

In view of the previous result, the following problem arises.

\begin{problem}
Does there exist a polyhedral Lindenstrauss space $E$ satisfying $B_E=\cconv\bigl(\mathrm{ext}(B_E)\bigr)$?
\end{problem}

\subsection{Cardinality of $\ext(B_E)$}

As pointed out above, the unit ball of a $(\mathrm{IV})$-poly\-hedral Banach space does not admit extreme points. However, the results presented above and those contained in \cite{DEPOLY,Schreier} show that it is conceivable that the unit ball of a polyhedral Banach space has many extreme points, in the sense that it is the closed	convex hull of its extreme points.

Another related problem could be the following: how many extreme points, in the sense of cardinality, the unit ball of a polyhedral Banach space could have? Let us point out that the separable polyhedral Banach spaces  constructed in \cite{DEPOLY,Schreier} are such that  the corresponding unit ball contains countably many extreme points. This fact can be directly verified or, alternatively, one can apply the following easy-to-prove observation, taking into account that these spaces are $(\mathrm{V})$-polyhedral (and hence $(\mathrm{VI})$-polyhedral).

\begin{observation}\label{obs: extisolated}
	Let $E$ be a $(\mathrm{VI})$-polyhedral Banach space, then each point in $\ext(B_E)$ is $\|\cdot\|$-isolated.
\end{observation}

This leads to the following problem, which, by the observation above, has obviously an affirmative answer in the case in which the space considered is $(\mathrm{VI})$-polyhedral.

\begin{problem}
	Does $E$ separable polyhedral imply $\ext(B_E)$ is countable?
More in general,  does $$\mathrm{card}(\ext(B_E))\leq\mathrm{dens}(E)$$ hold, whenever  $E$ is a polyhedral Banach space?
\end{problem}

\subsection*{Acknowledgement}

The research of the first author has been partially
supported by GNAMPA (INdAM -- Istituto Nazionale di Alta Matematica)  and by the Ministry for Science and Innovation, Spanish State Research Agency (Spain),  under project 
PID2020-112491GB-I00.


\begin{thebibliography}{WW}


\bibitem{Schreier} L.~Antunes, K.~Beanland and H.V.~Chu,
{\em On the geometry of higher order Schreier spaces},
Illinois J. Math. {\bf 65} (2021), 47--69.


\bibitem{BibleSmith} V. Bible and R.J. Smith,
{\em Smooth and polyhedral approximation in {B}anach spaces},
J. Math. Anal. Appl. {\bf 435} (2016),  1262--1272.





\bibitem{casinimiglierinapiaseckivesely}
 C. Casini, E.~Miglierina,  \L.~Piasecki and L.~Vesel\'{y},
{\em Rethinking polyhedrality for {L}indenstrauss spaces},
Israel J. Math. {\bf 216} (2016),  355--369.


	\bibitem{DEPOLY} 
C.A.~De Bernardi, Extreme points in polyhedral Banach spaces, Israel J. Math., \textbf{220} (2017), 547--557.




\bibitem{devillefonfhajek} R. Deville, V.P.~Fonf, and P.~H\'ajek,
{\em Analytic and polyhedral approximation of convex bodies in separable polyhedral Banach spaces},
Israel J. Math. {\bf 105} (1998),  139--154.


\bibitem{durpap} R.~Durier and P.L.~Papini,
{\em Polyhedral norms in an infinite-dimensional space}, Rocky Mountain J. Math. {\bf 23} (1993), 863--875.



\bibitem{fonfstrutt} V.P.~Fonf,
{\em Polyhedral Banach spaces},
Math. Notes USSR {\bf 30} (1981),  809--813

\bibitem{fonfstruttnew} V.P.~Fonf,
{\em On the boundary of a polyhedral Banach space},
Extracta Math. {\bf 15} (2000), 145--154.


\bibitem{InfPoly} V.P.~Fonf and L.~Vesel\'y,
{\em Infinite dimensional polyhedrality}, Canad. J. Math. {\bf 56} (2004), 472--494.

\bibitem{FonfLindVes} V.P.~Fonf, J.~Lindenstrauss and L.~Vesel\'y,
{\em Best approximation in polyhedral Banach spaces},  J. Approx. Theory {\bf 163} (2011), 1748--1771.


\bibitem{hand} V.P.~Fonf, J.~Lindenstrauss and R.R.~ Phelps,
{\em Infinite dimensional convexity, Handbook of the geometry of Banach spaces}, Vol. I, North-Holland, Amsterdam, 2001, pp. 599--670.








\bibitem{Klee60} V.~Klee,
{\em  Polyhedral sections of convex bodies}, Acta Math. {\bf 103} (1960), 243--267.

\bibitem{Klee59} V.~Klee,
{\em 
	Some characterizations of convex polyhedra}, Acta Math. {\bf 102} (1959), 79--107.



\bibitem{Lindenstrauss} J.~Lindenstrauss,
{\em Notes on Klee's paper: ``Polyhedral sections of convex bodies''},
Israel J. Math. {\bf 4} (1966), 235--242.

\bibitem{veselystrutt} L.~Vesel\'y,
{\em Boundary of polyhedral spaces: an alternative proof},
Extracta Math. {\bf 15} (2000), 213--217.

\end{thebibliography}
\end{document}